\documentclass[11pt]{article}

\usepackage{multirow}%
\usepackage{amsmath,amssymb,amsfonts}%
\usepackage{amsthm}%
\usepackage{mathrsfs}%
\usepackage[title]{appendix}%
\usepackage{xcolor}%
\usepackage{textcomp}%
\usepackage{manyfoot}%
\usepackage{booktabs}%
\usepackage{algorithm}%
\usepackage{algorithmicx}%
\usepackage{algpseudocode}%
\usepackage{listings}%
\usepackage{algorithm}
\usepackage{algpseudocode}
\usepackage{enumitem}
\usepackage{hyperref}

\hypersetup{
  colorlinks=true,
  linkcolor=blue,
  citecolor=blue,
  urlcolor=blue
}

\newtheorem{theorem}{Theorem}[section]
\newtheorem{lemma}[theorem]{Lemma}
\newtheorem{proposition}[theorem]{Proposition}
\newtheorem{corollary}[theorem]{Corollary}
\newtheorem{definition}[theorem]{Definition}

\newtheorem{problem}[theorem]{Problem}

\usepackage[style=authoryear, backend=biber]{biblatex} % Load the package
\title{Constrained Nonnegative Gram Feasibility is $\exists\mathbb{R}$-Complete}
\author{Angshul Majumdar, IIIT Delhi, India}

\begin{document}

\maketitle

\begin{abstract}
    
We study the computational complexity of constrained nonnegative Gram feasibility.
Given a partially specified symmetric matrix together with affine relations among
selected entries, the problem asks whether there exists a nonnegative matrix
$H \in \mathbb{R}_+^{n\times r}$ such that $W = HH^\top$ satisfies all specified
entries and affine constraints. Such factorizations arise naturally in structured
low-rank matrix representations and geometric embedding problems.

We prove that this feasibility problem is $\exists\mathbb{R}$-complete already
for rank $r=2$. The hardness result is obtained via a polynomial-time reduction
from the arithmetic feasibility problem \textsc{ETR-AMI}. The reduction exploits
a geometric encoding of arithmetic constraints within rank-$2$ nonnegative Gram
representations: by fixing anchor directions in $\mathbb{R}_+^2$ and representing
variables through vectors of the form $(x,1)$, addition and multiplication
constraints can be realized through inner-product relations.

Combined with the semialgebraic formulation of the feasibility conditions, this
establishes $\exists\mathbb{R}$-completeness. We further show that the hardness
extends to every fixed rank $r\ge 2$. Our results place constrained symmetric
nonnegative Gram factorization among the growing family of geometric feasibility
problems that are complete for the existential theory of the reals.

Finally, we discuss limitations of the result and highlight the open problem of
determining the complexity of unconstrained symmetric nonnegative factorization
feasibility.
\end{abstract}



\section{Introduction}
\label{sec:intro}

Low-rank matrix factorizations play a central role in several areas of algorithms,
optimization, and matrix theory. A particularly natural class of such
representations arises from Gram factorizations.
Given vectors $h_1,\dots,h_n \in \mathbb{R}^r$, the matrix
\[
W_{ij} = \langle h_i , h_j \rangle
\]
is symmetric and positive semidefinite. When the vectors are additionally
constrained to lie in the nonnegative orthant, $h_i \in \mathbb{R}_+^r$,
the resulting representation
\[
W = HH^\top, \qquad H \in \mathbb{R}_+^{n\times r},
\]
is called a \emph{nonnegative Gram factorization}.
Such factorizations are closely related to completely positive matrices,
copositive optimization, and nonnegative matrix factorization models
\cite{berman2003cp,bomze2018copositive,cohen1993nonnegativerank,
vavasis2009complexity,lee1999nmf}.

Despite the structural simplicity of this representation, the computational
complexity of deciding whether such factorizations exist remains poorly
understood. The decision problem naturally falls into the framework of
the \emph{existential theory of the reals} ($\exists\mathbb{R}$), the
complexity class capturing the feasibility of systems of polynomial
equalities and inequalities over the real numbers
\cite{schaefer2010complexity,cardinal2013survey}.
This class has emerged as a central object in the study of continuous
feasibility problems with algebraic structure.

Over the past decades, many geometric realization problems have been shown
to be complete for $\exists\mathbb{R}$. Classical examples include
pseudoline stretchability and oriented matroid realizability,
whose hardness follows from Mn\"ev's universality theorem
\cite{mnev1988universality,richter1996oriented}.
Subsequent work established $\exists\mathbb{R}$-completeness for a wide
range of geometric and combinatorial problems such as segment intersection graphs \cite{Foerster2026Thickness, Abrahamsen2025Embeddability} and geometric embedding problems
\cite{schaefer2010complexity,cardinal2013survey, Erickson2025Canonical}.
These results demonstrate that apparently simple geometric feasibility
questions can already encode arbitrary semialgebraic arithmetic structure.

Matrix factorization problems exhibit similarly rich computational
behavior. Nonnegative matrix factorization is known to be NP-hard
in general \cite{vavasis2009complexity}. Closely related notions such as
nonnegative rank and completely positive factorizations have been
studied extensively in matrix theory and optimization
\cite{cohen1993nonnegativerank,berman2003cp,bomze2018copositive}.
More recently, connections between matrix factorizations and convex
geometry have also been explored through notions such as PSD rank
and extended formulations \cite{fawzi2015psdrank}.
However, the precise complexity of low-rank nonnegative Gram realizations
from the viewpoint of existential-real complexity has received
comparatively little attention.

In this paper we study a constrained version of the nonnegative Gram
realization problem in low dimension.
The input consists of a partially specified symmetric matrix together
with affine constraints on selected entries.
The question is whether there exists a matrix
$H \in \mathbb{R}_+^{n\times 2}$ such that
\[
W = HH^\top
\]
on all specified entries and all affine side constraints are satisfied.
We refer to this problem as
\emph{constrained rank-2 nonnegative Gram feasibility}.

Our main result shows that this problem already captures the expressive
power of existential-real arithmetic.

\begin{theorem}
\label{thm:intro-main}
Constrained rank-2 nonnegative Gram feasibility is
$\exists\mathbb{R}$-complete.
\end{theorem}

The reduction is from the arithmetic normal form problem \textsc{ETR-AMI}.
The key observation is that rank-2 nonnegative Gram geometry admits a
simple arithmetic encoding. Two anchor rows force the canonical
directions $(1,0)$ and $(0,1)$ in $\mathbb{R}_+^2$.
Variable rows are then constrained to lie on the affine line $(x,1)$,
so that selected Gram entries recover variables linearly,
while pairwise inner products generate multiplicative relations via
\[
\langle (x,1),(y,1)\rangle = xy + 1 .
\]
Using this mechanism, addition, multiplication, and constant constraints
among real variables can be encoded by affine relations among entries
of the Gram matrix.

Our result shows that even in the smallest nontrivial dimension,
constrained nonnegative Gram realizations already possess the expressive
power of existential-real arithmetic. This places the problem within the
$\exists\mathbb{R}$ complexity landscape and highlights the algebraic
richness of low-rank nonnegative Gram geometry.
We also show that the hardness propagates to every fixed rank $r\ge 2$.

\paragraph{Organization of the paper.}
Section~\ref{sec:background} introduces the formal framework.
Section~\ref{sec:geometry} develops the rank-2 nonnegative Gram geometry.
Section~\ref{sec:reduction} presents the reduction from \textsc{ETR-AMI}.
Section~\ref{sec:consequences} derives the resulting complexity classification.
Section~\ref{sec:open} states several open problems, and
Section~\ref{sec:conclusion} concludes.

\section{Background and Preliminaries}
\label{sec:background}

This section fixes notation and recalls the formal complexity framework used
throughout the paper. We first review the class $\exists\mathbb{R}$, then
state the arithmetic normal form used in the reduction, and finally define
the constrained nonnegative Gram feasibility problem studied here.

\subsection{The class \texorpdfstring{$\exists\mathbb{R}$}{ER}}

The class $\exists\mathbb{R}$ consists of all decision problems that are
polynomial-time reducible to the existential theory of the reals. A standard
complete problem is the following: given multivariate polynomials
$p_1,\dots,p_s,q_1,\dots,q_t$ with integer coefficients, decide whether the
sentence
\[
\exists x \in \mathbb{R}^m \;:\;
\bigwedge_{i=1}^{s} p_i(x)=0
\;\wedge\;
\bigwedge_{j=1}^{t} q_j(x)\ge 0
\]
is true \cite{schaefer2010complexity,cardinal2013survey}. Equivalently,
$\exists\mathbb{R}$ captures feasibility questions for semialgebraic sets
specified by polynomial equations and inequalities.

The class $\exists\mathbb{R}$ arises naturally in geometric and algebraic
realizability problems. Classical examples include pseudoline stretchability
and oriented matroid realizability, whose hardness is rooted in Mn\"ev's
universality theorem \cite{mnev1988universality,richter1996oriented}.
A large number of computational geometry problems are now known to be
$\exists\mathbb{R}$-complete, including several geometric embedding and
intersection-graph realizability problems
\cite{schaefer2010complexity,cardinal2013survey}. Thus
$\exists\mathbb{R}$-hardness is a standard notion of intractability for
continuous feasibility problems with algebraic structure. Recent results further clarify the power of $\exists\mathbb{R}$ through
tensor-rank problems~\cite{Schaefer2018Tensor}, fixed-point
and Nash-equilibrium questions~\cite{Schaefer2017Fixed}, and
extensions of the class itself~\cite{Schaefer2024Beyond}.

It is known that
\[
\mathrm{NP} \subseteq \exists\mathbb{R} \subseteq \mathrm{PSPACE},
\]
and both inclusions are widely believed to be strict
\cite{schaefer2010complexity}. For this reason, $\exists\mathbb{R}$ has become
the natural complexity class for exact feasibility questions involving real
variables and polynomial constraints.

\subsection{Arithmetic normal form}

To prove hardness we reduce from a restricted arithmetic normal form for
$\exists\mathbb{R}$.

\begin{definition}[ETR-AMI]
\label{def:etrami}
An instance of \textsc{ETR-AMI} consists of variables
\[
x_1,\dots,x_n \in \mathbb{R}_{\ge 0}
\]
together with constraints of the following three forms:
\[
x_i = 1,\qquad
x_i = x_j + x_k,\qquad
x_i = x_j x_k.
\]
The decision problem asks whether there exists a nonnegative real assignment
satisfying all constraints.
\end{definition}

Standard normal-form results for the existential theory of the reals imply
that systems of this type can be used in place of more general existential
formulas in $\exists\mathbb{R}$-hardness reductions
\cite{schaefer2010complexity,cardinal2013survey}. Such restricted arithmetic
systems are convenient because they isolate the basic algebraic operations
needed to encode semialgebraic feasibility. In particular, once constants,
addition, and multiplication can be simulated, one obtains the expressive
power required for general existential-real arithmetic
\cite{schaefer2010complexity,cardinal2013survey}.

\subsection{Nonnegative Gram factorizations}

Let $W \in \mathbb{R}^{n\times n}$ be symmetric. A rank-$r$ Gram factorization
of $W$ is a representation
\[
W = HH^\top
\]
for some matrix $H \in \mathbb{R}^{n\times r}$. If the rows of $H$ are denoted
by $h_1,\dots,h_n \in \mathbb{R}^r$, this is equivalent to the identities
\[
W_{ij} = \langle h_i,h_j\rangle
\qquad \text{for all } i,j.
\]

If in addition each row vector lies in the nonnegative orthant,
\[
h_i \in \mathbb{R}_+^r,
\]
then $W = HH^\top$ is called a \emph{nonnegative Gram factorization}. Such
factorizations are closely related to completely positive matrices, copositive
optimization, and nonnegative matrix factorization
\cite{berman2003cp,bomze2018copositive,cohen1993nonnegativerank,
laurent2015cp,vavasis2009complexity}. More broadly, low-rank matrix factorizations
and their complexity have been studied in several neighboring settings,
including exact and approximate nonnegative factorization and positive
semidefinite rank \cite{vavasis2009complexity,arora2012nmf,fawzi2015psdrank}.

\subsection{Constrained nonnegative Gram feasibility}

We now define the decision problem studied in this paper.

\begin{definition}[Constrained rank-$r$ nonnegative Gram feasibility]
\label{def:gramfeas}
An instance consists of the following data:
\begin{enumerate}[label=(\roman*)]
\item a finite index set of specified pairs $(i,j)$ together with prescribed
rational values $\widehat{W}_{ij}$ for those entries of a symmetric matrix
$W \in \mathbb{R}^{n\times n}$;
\item a finite family of affine constraints over selected entries of $W$;
\item a target rank parameter $r$.
\end{enumerate}
The question is whether there exists a matrix
\[
H \in \mathbb{R}_+^{n\times r}
\]
such that
\[
(HH^\top)_{ij} = \widehat{W}_{ij}
\qquad \text{for all specified pairs } (i,j),
\]
and all affine side constraints are satisfied after substituting
\[
W_{ij} = (HH^\top)_{ij}.
\]
\end{definition}

Thus the only existentially quantified variables are the entries of $H$.
The matrix $W$ itself serves only as a convenient notation for specified Gram
entries and for affine relations among them.

In this paper we focus on the first nontrivial case $r=2$.

\subsection{Membership in \texorpdfstring{$\exists\mathbb{R}$}{ER}}

The preceding problem lies in $\exists\mathbb{R}$.

\begin{proposition}
\label{prop:er-membership}
Constrained rank-$r$ nonnegative Gram feasibility belongs to
$\exists\mathbb{R}$.
\end{proposition}

\begin{proof}
Let the entries of the unknown matrix $H$ be denoted by
\[
h_{ik}, \qquad 1\le i\le n,\; 1\le k\le r.
\]
These are the only quantified real variables.

For every specified pair $(i,j)$ with prescribed rational value
$\widehat{W}_{ij}$, the Gram requirement becomes the polynomial equation
\[
\sum_{k=1}^{r} h_{ik}h_{jk} = \widehat{W}_{ij}.
\]
Nonnegativity of the realization is expressed by the polynomial inequalities
\[
h_{ik} \ge 0
\qquad \text{for all } i,k.
\]

Finally, each affine side constraint on selected entries of $W$ has the form
\[
\sum_{\ell=1}^{m} a_\ell W_{i_\ell j_\ell} = b
\]
with rational coefficients $a_\ell,b$. After substituting
\[
W_{ij}=\sum_{k=1}^{r} h_{ik}h_{jk},
\]
this becomes a polynomial equation in the variables $h_{ik}$.

Hence the entire feasibility problem can be written as a conjunction of
polynomial equalities and inequalities over the reals. Therefore it defines
an instance of the existential theory of the reals and belongs to
$\exists\mathbb{R}$ \cite{schaefer2010complexity,cardinal2013survey}.
\end{proof}

\section{Rank-2 Nonnegative Gram Geometry}
\label{sec:geometry}

In this section we isolate the geometric mechanism underlying the reduction.
We work with rank-2 nonnegative Gram realizations and show that, once two
distinguished anchor rows are fixed, the remaining rows can be used to encode
nonnegative real variables and their arithmetic relations.

Symmetric matrices admitting factorizations of the form $W=HH^\top$ with
$H\ge 0$ are completely positive in the sense of matrix theory
\cite{berman2003cp,laurent2015cp}. Thus the present construction may be
viewed as a rank-2 completely positive realization with additional affine
constraints. The point of this section is that, already in rank $2$, the
nonnegative orthant has enough structure to support a simple arithmetic
encoding \cite{berman2003cp,bomze2018copositive,laurent2015cp}.

\subsection{Rank-2 nonnegative Gram realizations}

Let $W$ be a symmetric matrix admitting a rank-2 nonnegative Gram
factorization. Thus there exists
\[
H \in \mathbb{R}_+^{n\times 2}
\]
such that
\[
W = HH^\top.
\]
If the rows of $H$ are denoted by $h_1,\dots,h_n \in \mathbb{R}_+^2$, then
\[
W_{ij} = \langle h_i,h_j\rangle
\qquad \text{for all } i,j.
\]
Throughout this section, all realizations are assumed to be of this form.

\begin{definition}[Anchor rows]
\label{def:anchors}
Two distinguished indices $e$ and $f$ are called \emph{anchors} if the
corresponding Gram entries satisfy
\[
W_{ee}=1,\qquad W_{ff}=1,\qquad W_{ef}=0.
\]
\end{definition}

The anchor constraints force the corresponding vectors to coincide with the
two coordinate directions in $\mathbb{R}_+^2$.

\begin{lemma}[Anchor rigidity]
\label{lem:anchor-rigidity}
Suppose $W$ admits a rank-2 nonnegative Gram realization satisfying the anchor
constraints of Definition~\ref{def:anchors}. Then, up to permutation of the
two coordinates,
\[
h_e=(1,0),\qquad h_f=(0,1).
\]
\end{lemma}

\begin{proof}
Write
\[
h_e=(a,b),\qquad h_f=(c,d)
\]
with $a,b,c,d\ge 0$. The Gram constraints imply
\[
a^2+b^2 = 1,\qquad c^2+d^2 = 1,\qquad ac+bd = 0.
\]
Since all quantities are nonnegative, the equality $ac+bd=0$ implies
\[
ac=0,\qquad bd=0.
\]
Hence the supports of $h_e$ and $h_f$ are disjoint. Because each vector has
unit norm and both lie in $\mathbb{R}_+^2$, it follows that one of them equals
$(1,0)$ and the other equals $(0,1)$. This is exactly the claim, up to
swapping the two coordinates.
\end{proof}

Thus the anchor rows canonically determine a coordinate system in the
nonnegative orthant.

\subsection{Variable rows}

We next introduce rows that encode scalar variables.

\begin{definition}[Variable rows]
\label{def:variable-rows}
For each arithmetic variable we introduce an index $u$ and impose the
constraint
\[
W_{fu}=1.
\]
Such an index will be called a \emph{variable row}.
\end{definition}

The effect of this constraint is immediate once the anchors are fixed.

\begin{lemma}[Variable encoding]
\label{lem:variable-encoding}
Assume the anchor normalization
\[
h_e=(1,0),\qquad h_f=(0,1).
\]
Let $u$ be a variable row satisfying $W_{fu}=1$. Then
\[
h_u=(x_u,1)
\]
for a uniquely determined scalar $x_u\in \mathbb{R}_{\ge 0}$.
\end{lemma}

\begin{proof}
Write
\[
h_u=(a,b)
\]
with $a,b\ge 0$. Since $h_f=(0,1)$, the constraint $W_{fu}=1$ gives
\[
1 = \langle h_f,h_u\rangle = b.
\]
Hence $b=1$. Setting $x_u:=a$, we obtain
\[
h_u=(x_u,1),
\]
and uniqueness is immediate.
\end{proof}

Therefore every variable row lies on the affine line
\[
\{(x,1):x\in\mathbb{R}_{\ge 0}\}\subset \mathbb{R}_+^2.
\]
This line will serve as the coordinate space for the arithmetic variables in
the reduction.

\subsection{Extraction of coordinates}

Once a variable row has the form $(x_u,1)$, its first coordinate is directly
recoverable from the Gram matrix.

\begin{lemma}[Coordinate extraction]
\label{lem:coordinate-extraction}
Under the assumptions of Lemma~\ref{lem:variable-encoding},
\[
W_{eu}=x_u.
\]
\end{lemma}

\begin{proof}
Since $h_e=(1,0)$ and $h_u=(x_u,1)$,
\[
W_{eu}=\langle h_e,h_u\rangle
      =\langle (1,0),(x_u,1)\rangle
      =x_u.
\]
\end{proof}

Thus specified Gram entries against the anchor row $e$ can be used to read off
the encoded variables linearly.

\subsection{Multiplicative structure}

The crucial feature of the construction is that inner products between
variable rows automatically generate multiplication.

\begin{lemma}[Multiplication identity]
\label{lem:multiplication-identity}
Let $u$ and $v$ be variable rows encoding scalars $x_u,x_v\in\mathbb{R}_{\ge 0}$.
Then
\[
W_{uv} = x_u x_v + 1.
\]
\end{lemma}

\begin{proof}
By Lemma~\ref{lem:variable-encoding},
\[
h_u=(x_u,1),\qquad h_v=(x_v,1).
\]
Therefore
\[
W_{uv}
= \langle h_u,h_v\rangle
= \langle (x_u,1),(x_v,1)\rangle
= x_u x_v + 1.
\]
\end{proof}

This identity is the basic source of arithmetic expressivity in the reduction.

\subsection{Arithmetic realization}

Combining the preceding lemmas yields a direct translation from arithmetic
constraints to affine relations among Gram entries.

\begin{proposition}[Arithmetic realization in rank $2$]
\label{prop:arithmetic-realization}
Assume a rank-2 nonnegative Gram realization with anchors
\[
h_e=(1,0),\qquad h_f=(0,1),
\]
and let $u,v,w$ be variable rows encoding $x_u,x_v,x_w\in\mathbb{R}_{\ge 0}$.
Then the following equivalences hold:
\begin{align*}
x_u = 1
&\iff W_{eu}=1,\\[0.5ex]
x_u + x_v = x_w
&\iff W_{eu}+W_{ev}-W_{ew}=0,\\[0.5ex]
x_u x_v = x_w
&\iff W_{uv}-W_{ew}-1=0.
\end{align*}
\end{proposition}

\begin{proof}
The first equivalence follows immediately from
Lemma~\ref{lem:coordinate-extraction}. For the second, again by
Lemma~\ref{lem:coordinate-extraction},
\[
W_{eu}+W_{ev}-W_{ew}
= x_u+x_v-x_w,
\]
so the displayed equality holds if and only if $x_u+x_v=x_w$.

For the third, Lemma~\ref{lem:multiplication-identity} gives
\[
W_{uv}=x_u x_v+1,
\]
while Lemma~\ref{lem:coordinate-extraction} gives
\[
W_{ew}=x_w.
\]
Hence
\[
W_{uv}-W_{ew}-1 = x_u x_v - x_w,
\]
and the claim follows.
\end{proof}

Proposition~\ref{prop:arithmetic-realization} shows that constants,
addition, and multiplication can all be represented by affine relations among
entries of a rank-2 nonnegative Gram matrix. This is precisely the structure
required for the reduction from \textsc{ETR-AMI} in the next section.

\section{Reduction from \textsc{ETR-AMI}}
\label{sec:reduction}

In this section we prove the main lower bound by giving a polynomial-time
reduction from the arithmetic normal form introduced in
Definition~\ref{def:etrami} to constrained rank-2 nonnegative Gram feasibility.
The reduction uses the geometric encoding developed in
Section~\ref{sec:geometry}.

Let
\[
\mathcal{I}
\]
be an instance of \textsc{ETR-AMI} consisting of variables
\[
x_1,\dots,x_n \in \mathbb{R}_{\ge 0}
\]
and a finite family of constraints of the forms
\[
x_i = 1,\qquad x_i = x_j + x_k,\qquad x_i = x_j x_k.
\]
As discussed in Section~\ref{sec:background}, arithmetic systems of this type
arise from standard normal-form transformations for the existential theory of
the reals and are routinely used as source problems in
\(\exists\mathbb{R}\)-hardness reductions
\cite{schaefer2010complexity,cardinal2013survey}.

We construct from \(\mathcal{I}\) an instance of
Definition~\ref{def:gramfeas} with target rank \(r=2\).

\subsection{Construction}
\label{subsec:construction}

The output instance consists of a set of specified Gram entries, a family of
affine side constraints, and rank parameter \(r=2\), exactly as allowed by
Definition~\ref{def:gramfeas}. The existential variables of the target
instance are the entries of a realization matrix
\[
H \in \mathbb{R}_+^{(n+2)\times 2}.
\]

\paragraph{Anchor rows.}
Introduce two distinguished indices \(e\) and \(f\). Specify the Gram entries
\[
\widehat{W}_{ee}=1,\qquad \widehat{W}_{ff}=1,\qquad \widehat{W}_{ef}=0.
\]
By Lemma~\ref{lem:anchor-rigidity}, in every feasible realization the
corresponding row vectors satisfy
\[
h_e=(1,0),\qquad h_f=(0,1),
\]
up to permutation of the two coordinates. Since all constraints in the
construction are expressed solely in terms of Gram entries, and Gram entries
are invariant under a simultaneous permutation of the two coordinates of all
row vectors, such a permutation yields an equivalent feasible realization.
Hence we may fix the coordinate order once and for
all so that the canonical normalization above holds.

\paragraph{Variable rows.}
For each arithmetic variable \(x_i\), introduce one index \(u_i\). For every
such index, specify the Gram entry
\[
\widehat{W}_{f u_i}=1.
\]
By Lemma~\ref{lem:variable-encoding}, any feasible realization must then have
\[
h_{u_i}=(x_i,1)
\]
for a uniquely determined scalar \(x_i \in \mathbb{R}_{\ge 0}\).

\paragraph{Constant constraints.}
For every source constraint \(x_i=1\), add the affine side constraint
\[
W_{e u_i}=1.
\]
By Lemma~\ref{lem:coordinate-extraction}, this is equivalent to \(x_i=1\).

\paragraph{Addition constraints.}
For every source constraint
\[
x_i=x_j+x_k,
\]
add the affine side constraint
\[
W_{e u_j}+W_{e u_k}-W_{e u_i}=0.
\]
By Lemma~\ref{lem:coordinate-extraction}, this is equivalent to
\[
x_j+x_k-x_i=0.
\]

\paragraph{Multiplication constraints.}
For every source constraint
\[
x_i=x_jx_k,
\]
add the affine side constraint
\[
W_{u_j u_k}-W_{e u_i}-1=0.
\]
By Lemmas~\ref{lem:multiplication-identity} and
\ref{lem:coordinate-extraction}, this is equivalent to
\[
x_jx_k-x_i=0.
\]

This completes the construction.

\subsection{Encoding size and complexity}
\label{subsec:size}

We now verify that the construction has polynomial size.

\begin{lemma}
\label{lem:size}
The mapping \(\mathcal{I}\mapsto \mathcal{J}\), from an instance
\(\mathcal{I}\) of \textsc{ETR-AMI} to the constructed instance
\(\mathcal{J}\) of constrained rank-2 nonnegative Gram feasibility, is
computable in time polynomial in the size of \(\mathcal{I}\).
\end{lemma}

\begin{proof}
Suppose \(\mathcal{I}\) has \(n\) variables and \(m\) arithmetic constraints.
The target instance introduces exactly \(n+2\) indices: two anchor indices
\(e,f\), and one index \(u_i\) for each source variable \(x_i\).

The number of specified Gram entries is linear in \(n\): three anchor entries
\[
\widehat{W}_{ee},\ \widehat{W}_{ff},\ \widehat{W}_{ef},
\]
together with one entry \(\widehat{W}_{f u_i}\) for each \(i=1,\dots,n\).

Each arithmetic constraint of the source instance contributes exactly one
affine side constraint in the target instance. Hence the total number of
affine constraints is \(m\).

Moreover, every coefficient appearing in the target description belongs to
\[
\{-1,0,1\},
\]
and every constant term belongs to
\[
\{0,1\}.
\]
Thus the bit-length of every numeric coefficient is \(O(1)\). It follows that
the full encoding length of the target instance is \(O(n+m)\), and the
construction can be carried out in polynomial time.
\end{proof}

\subsection{Soundness}
\label{subsec:soundness}

We next show that every satisfying assignment of the arithmetic instance
produces a feasible nonnegative Gram realization.

\begin{lemma}[Soundness]
\label{lem:soundness}
If the \textsc{ETR-AMI} instance \(\mathcal{I}\) is satisfiable, then the
constructed constrained rank-2 nonnegative Gram instance \(\mathcal{J}\) is
feasible.
\end{lemma}

\begin{proof}
Assume that \(\mathcal{I}\) has a satisfying assignment
\[
x_1,\dots,x_n \in \mathbb{R}_{\ge 0}.
\]
Define vectors in \(\mathbb{R}_+^2\) by
\[
h_e=(1,0),\qquad h_f=(0,1),
\]
and, for each variable \(x_i\),
\[
h_{u_i}=(x_i,1).
\]
Since each \(x_i\ge 0\), all these vectors lie in \(\mathbb{R}_+^2\).

Let \(H\) be the matrix whose rows are the vectors just defined, and let
\[
W = HH^\top.
\]
We verify that this realization satisfies all specified entries and affine side
constraints of \(\mathcal{J}\).

First, the anchor constraints hold:
\[
W_{ee}=\langle (1,0),(1,0)\rangle = 1,
\]
\[
W_{ff}=\langle (0,1),(0,1)\rangle = 1,
\]
\[
W_{ef}=\langle (1,0),(0,1)\rangle = 0.
\]

Second, for each variable row \(u_i\),
\[
W_{f u_i}
=
\langle (0,1),(x_i,1)\rangle
=
1,
\]
so all variable-row constraints are satisfied.

Third, consider a source constraint \(x_i=1\). The corresponding affine side
constraint in \(\mathcal{J}\) is \(W_{e u_i}=1\). Since
\[
W_{e u_i}
=
\langle (1,0),(x_i,1)\rangle
=
x_i,
\]
this holds because the assignment satisfies \(x_i=1\).

Fourth, consider a source addition constraint \(x_i=x_j+x_k\). The
corresponding affine side constraint is
\[
W_{e u_j}+W_{e u_k}-W_{e u_i}=0.
\]
Using the definition of the vectors,
\[
W_{e u_j}+W_{e u_k}-W_{e u_i}
=
x_j+x_k-x_i
=
0.
\]

Finally, consider a source multiplication constraint \(x_i=x_jx_k\). The
corresponding affine side constraint is
\[
W_{u_j u_k}-W_{e u_i}-1=0.
\]
Now
\[
W_{u_j u_k}
=
\langle (x_j,1),(x_k,1)\rangle
=
x_jx_k+1,
\]
and
\[
W_{e u_i}=x_i.
\]
Hence
\[
W_{u_j u_k}-W_{e u_i}-1
=
(x_jx_k+1)-x_i-1
=
x_jx_k-x_i
=
0.
\]

Therefore all specified entries and all affine side constraints in
\(\mathcal{J}\) are satisfied. Thus \(\mathcal{J}\) is feasible.
\end{proof}

\subsection{Completeness}
\label{subsec:completeness}

We now show that every feasible realization of the target instance yields a
satisfying arithmetic assignment.

\begin{lemma}[Completeness]
\label{lem:completeness}
If the constructed constrained rank-2 nonnegative Gram instance
\(\mathcal{J}\) is feasible, then the original \textsc{ETR-AMI} instance
\(\mathcal{I}\) is satisfiable.
\end{lemma}

\begin{proof}
Assume that \(\mathcal{J}\) has a feasible realization
\[
H \in \mathbb{R}_+^{(n+2)\times 2}.
\]
Let the corresponding row vectors be denoted by
\(h_e,h_f,h_{u_1},\dots,h_{u_n}\).

By the anchor constraints
\[
W_{ee}=1,\qquad W_{ff}=1,\qquad W_{ef}=0,
\]
Lemma~\ref{lem:anchor-rigidity} implies that, up to permutation of the two
coordinates,
\[
h_e=(1,0),\qquad h_f=(0,1).
\]
As in the construction, because a simultaneous permutation of coordinates
preserves all Gram entries, we may fix the coordinate order so that this
normalization holds.

For each variable row \(u_i\), the specified entry \(W_{f u_i}=1\) and
Lemma~\ref{lem:variable-encoding} imply that
\[
h_{u_i}=(x_i,1)
\]
for a uniquely determined scalar \(x_i\in\mathbb{R}_{\ge 0}\).

We claim that the extracted numbers \(x_1,\dots,x_n\) satisfy every source
constraint of \(\mathcal{I}\).

\paragraph{Constant constraints.}
Suppose \(\mathcal{I}\) contains the constraint \(x_i=1\). Then \(\mathcal{J}\)
contains the affine side constraint
\[
W_{e u_i}=1.
\]
By Lemma~\ref{lem:coordinate-extraction},
\[
x_i=W_{e u_i}=1.
\]

\paragraph{Addition constraints.}
Suppose \(\mathcal{I}\) contains the constraint
\[
x_i=x_j+x_k.
\]
Then \(\mathcal{J}\) contains
\[
W_{e u_j}+W_{e u_k}-W_{e u_i}=0.
\]
Applying Lemma~\ref{lem:coordinate-extraction} to each term yields
\[
x_j+x_k-x_i=0,
\]
hence \(x_i=x_j+x_k\).

\paragraph{Multiplication constraints.}
Suppose \(\mathcal{I}\) contains the constraint
\[
x_i=x_jx_k.
\]
Then \(\mathcal{J}\) contains
\[
W_{u_j u_k}-W_{e u_i}-1=0.
\]
By Lemma~\ref{lem:multiplication-identity},
\[
W_{u_j u_k}=x_jx_k+1,
\]
and by Lemma~\ref{lem:coordinate-extraction},
\[
W_{e u_i}=x_i.
\]
Therefore
\[
(x_jx_k+1)-x_i-1=0,
\]
which simplifies to
\[
x_i=x_jx_k.
\]

Since every source constraint is satisfied, the extracted nonnegative tuple
\[
(x_1,\dots,x_n)
\]
is a satisfying assignment of \(\mathcal{I}\). Thus \(\mathcal{I}\) is
satisfiable.
\end{proof}

\subsection{Main complexity consequence}
\label{subsec:main-consequence}

We can now state the main theorem.

\begin{theorem}
\label{thm:hardness}
Constrained rank-2 nonnegative Gram feasibility is
\(\exists\mathbb{R}\)-hard.
\end{theorem}

\begin{proof}
By Lemma~\ref{lem:size}, the reduction described above is computable in
polynomial time. By Lemma~\ref{lem:soundness}, every satisfying assignment of
the source instance \(\mathcal{I}\) yields a feasible realization of the
target instance \(\mathcal{J}\). Conversely, by
Lemma~\ref{lem:completeness}, every feasible realization of \(\mathcal{J}\)
yields a satisfying assignment of \(\mathcal{I}\).

Therefore the reduction is correct. Since arithmetic systems of the type used
in Definition~\ref{def:etrami} arise from standard normal-form reductions for
the existential theory of the reals and serve as canonical source problems in
\(\exists\mathbb{R}\)-hardness proofs
\cite{schaefer2010complexity,cardinal2013survey}, it follows that constrained
rank-2 nonnegative Gram feasibility is \(\exists\mathbb{R}\)-hard.
\end{proof}

\section{Consequences and Complexity Classification}
\label{sec:consequences}

In this section we record formal consequences of the reduction established in
Section~\ref{sec:reduction}. Throughout, all statements are derived directly
from the results already proved in Sections~\ref{sec:background},
\ref{sec:geometry}, and \ref{sec:reduction}.

\subsection{The rank-$2$ classification}

We first combine the membership result of
Proposition~\ref{prop:er-membership} with the hardness result of
Theorem~\ref{thm:hardness}.

\begin{theorem}
\label{thm:classification-r2}
Constrained rank-$2$ nonnegative Gram feasibility is
$\exists\mathbb{R}$-complete.
\end{theorem}

\begin{proof}
By Proposition~\ref{prop:er-membership}, constrained rank-$r$ nonnegative
Gram feasibility belongs to $\exists\mathbb{R}$ for every rank parameter
$r$. In particular, the rank-$2$ case belongs to $\exists\mathbb{R}$.

By Theorem~\ref{thm:hardness}, constrained rank-$2$ nonnegative Gram
feasibility is $\exists\mathbb{R}$-hard. Combining these two statements yields
the claimed $\exists\mathbb{R}$-completeness.
\end{proof}

Thus the constrained rank-$2$ problem belongs to the class of
$\exists\mathbb{R}$-complete semialgebraic feasibility problems studied in
computational geometry and real algebraic complexity
\cite{schaefer2010complexity,cardinal2013survey}. For comparison, hardness
results for nonnegative matrix factorization itself are also known; in
particular, determining the nonnegative rank of a matrix is NP-hard
\cite{vavasis2009complexity}.

\subsection{Arithmetic representability}

The reduction of Section~\ref{sec:reduction} may be restated as an exact
representability consequence.

\begin{corollary}
\label{cor:arithmetic-representability}
Let
\[
\mathcal{S}=\{x_1,\dots,x_n\}
\]
be a finite set of variables ranging over \(\mathbb{R}_{\ge 0}\), and let
\(\mathcal{C}\) be a finite family of constraints of the forms
\[
x_i=1,\qquad x_i=x_j+x_k,\qquad x_i=x_jx_k.
\]
Then one can construct, in polynomial time, an instance of constrained
rank-$2$ nonnegative Gram feasibility such that the constructed instance is
feasible if and only if there exists a nonnegative assignment to the variables
in \(\mathcal{S}\) satisfying all constraints in \(\mathcal{C}\).
\end{corollary}

\begin{proof}
This is exactly the construction of
Section~\ref{subsec:construction}. Polynomial-time computability follows from
Lemma~\ref{lem:size}. The ``if'' direction follows from
Lemma~\ref{lem:soundness}, and the ``only if'' direction follows from
Lemma~\ref{lem:completeness}. Hence the constructed instance has the stated
equivalence property.
\end{proof}

\subsection{Propagation to higher ranks}

The reduction for rank \(2\) extends immediately to every larger fixed rank.

\begin{lemma}
\label{lem:rank-lift}
Let \(r\ge 2\). Every feasible instance of constrained rank-$2$
nonnegative Gram feasibility is also feasible as an instance of
constrained rank-$r$ nonnegative Gram feasibility.
\end{lemma}

\begin{proof}
Let \(r\ge 2\), and consider a feasible instance of constrained rank-$2$
nonnegative Gram feasibility. Let
\[
h_i=(a_i,b_i)\in\mathbb{R}_+^2
\]
be a feasible realization. Define
\[
\iota(h_i)=(a_i,b_i,0,\dots,0)\in\mathbb{R}_+^r.
\]
Then for all \(i,j\),
\[
\langle \iota(h_i),\iota(h_j)\rangle
=
a_i a_j+b_i b_j
=
\langle h_i,h_j\rangle.
\]
Hence all Gram entries are preserved. Since the constraints of
Definition~\ref{def:gramfeas} depend only on Gram entries, the embedded
vectors satisfy the same specified entries and affine side constraints.
Therefore the same instance is feasible in rank \(r\).
\end{proof}

\begin{theorem}
\label{thm:classification-fixed-r}
For every fixed integer \(r\ge 2\), constrained rank-$r$ nonnegative Gram
feasibility is \(\exists\mathbb{R}\)-complete.
\end{theorem}

\begin{proof}
Fix \(r\ge 2\). By Proposition~\ref{prop:er-membership}, constrained
rank-$r$ nonnegative Gram feasibility belongs to \(\exists\mathbb{R}\).

To prove hardness, reduce from constrained rank-$2$ nonnegative Gram
feasibility by the identity map on instances. If an instance is feasible
in rank \(2\), then by Lemma~\ref{lem:rank-lift} it is feasible in rank \(r\).
Hence every yes-instance of the rank-$2$ problem is a yes-instance of the
rank-$r$ problem. Since rank-$2$ feasibility is \(\exists\mathbb{R}\)-hard
by Theorem~\ref{thm:hardness}, the rank-$r$ problem is
\(\exists\mathbb{R}\)-hard as well. Combining hardness with membership gives
the claim.
\end{proof}

\subsection{Propagation to higher fixed ranks}

The rank-$2$ construction extends to every larger fixed rank by adding
additional anchor rows that force all arithmetic vectors to lie in the
two-dimensional coordinate plane spanned by the first two basis vectors.

\begin{lemma}[Anchor frame rigidity]
\label{lem:anchor-frame-rigidity}
Let \(r\ge 2\), and suppose a symmetric matrix \(W\) admits a rank-$r$
nonnegative Gram realization with distinguished indices
\[
e_1,\dots,e_r
\]
satisfying
\[
W_{e_a e_b}=\delta_{ab}
\qquad \text{for all } a,b\in\{1,\dots,r\},
\]
where \(\delta_{ab}\) denotes the Kronecker delta. Then, up to a common
permutation of the \(r\) coordinates,
\[
h_{e_a}=e_a^{(r)}
\qquad \text{for all } a=1,\dots,r,
\]
where \(e_a^{(r)}\) is the \(a\)-th standard basis vector of
\(\mathbb{R}^r\).
\end{lemma}

\begin{proof}
Write
\[
h_{e_a}=(h_{e_a,1},\dots,h_{e_a,r})\in\mathbb{R}_+^r.
\]
The Gram constraints imply
\[
\langle h_{e_a},h_{e_a}\rangle = 1
\qquad \text{for all } a,
\]
and
\[
\langle h_{e_a},h_{e_b}\rangle = 0
\qquad \text{for all } a\neq b.
\]
Since all coordinates are nonnegative, the equality
\[
\langle h_{e_a},h_{e_b}\rangle
=
\sum_{t=1}^r h_{e_a,t} h_{e_b,t}
=
0
\]
implies
\[
h_{e_a,t} h_{e_b,t}=0
\qquad \text{for every } t=1,\dots,r.
\]
Thus the supports of the \(r\) vectors \(h_{e_1},\dots,h_{e_r}\) are pairwise
disjoint.

Each \(h_{e_a}\) is nonzero, since it has unit norm. Hence we have \(r\)
pairwise disjoint nonempty supports contained in the \(r\)-element set
\(\{1,\dots,r\}\). Therefore each support must consist of exactly one
coordinate, and these \(r\) singleton supports must form a partition of
\(\{1,\dots,r\}\). Since each vector has norm \(1\), the unique nonzero
coordinate in each \(h_{e_a}\) must equal \(1\). After permuting coordinates,
we obtain
\[
h_{e_a}=e_a^{(r)}
\qquad \text{for all } a=1,\dots,r,
\]
as claimed.
\end{proof}

We can now prove the higher-rank analogue of the main hardness theorem.

\begin{theorem}
\label{thm:hardness-fixed-r}
For every fixed integer \(r\ge 2\), constrained rank-$r$ nonnegative Gram
feasibility is \(\exists\mathbb{R}\)-hard.
\end{theorem}

\begin{proof}
Fix \(r\ge 2\). We reduce from \textsc{ETR-AMI}
(Definition~\ref{def:etrami}).

Let
\[
\mathcal{I}
\]
be an instance of \textsc{ETR-AMI} with variables
\[
x_1,\dots,x_n\in\mathbb{R}_{\ge 0}
\]
and constraints of the forms
\[
x_i=1,\qquad x_i=x_j+x_k,\qquad x_i=x_jx_k.
\]

We construct an instance of constrained rank-$r$ nonnegative Gram feasibility
as follows.

\paragraph{Anchor rows.}
Introduce \(r\) distinguished indices
\[
e_1,\dots,e_r
\]
and specify the Gram entries
\[
\widehat{W}_{e_a e_b}=\delta_{ab}
\qquad \text{for all } a,b\in\{1,\dots,r\}.
\]
By Lemma~\ref{lem:anchor-frame-rigidity}, every feasible realization may be
normalized so that
\[
h_{e_a}=e_a^{(r)}
\qquad \text{for all } a=1,\dots,r.
\]

\paragraph{Variable rows.}
For each arithmetic variable \(x_i\), introduce one index \(u_i\).
Impose the specified Gram entry
\[
\widehat{W}_{e_2 u_i}=1,
\]
and, for every \(t=3,\dots,r\), impose the specified Gram entry
\[
\widehat{W}_{e_t u_i}=0.
\]

We claim that in every feasible realization,
\[
h_{u_i}=(x_i,1,0,\dots,0)
\]
for a uniquely determined scalar \(x_i\in\mathbb{R}_{\ge 0}\).
Indeed, write
\[
h_{u_i}=(a_{i,1},a_{i,2},\dots,a_{i,r})\in\mathbb{R}_+^r.
\]
Since \(h_{e_2}=e_2^{(r)}\), the constraint \(\widehat{W}_{e_2 u_i}=1\) gives
\[
a_{i,2}=1.
\]
Likewise, for each \(t=3,\dots,r\), the constraint
\(\widehat{W}_{e_t u_i}=0\) gives
\[
a_{i,t}=0.
\]
Setting \(x_i:=a_{i,1}\), we obtain
\[
h_{u_i}=(x_i,1,0,\dots,0),
\]
as claimed.

\paragraph{Constant constraints.}
For every source constraint \(x_i=1\), impose the affine side constraint
\[
W_{e_1 u_i}=1.
\]

\paragraph{Addition constraints.}
For every source constraint \(x_i=x_j+x_k\), impose the affine side constraint
\[
W_{e_1 u_j}+W_{e_1 u_k}-W_{e_1 u_i}=0.
\]

\paragraph{Multiplication constraints.}
For every source constraint \(x_i=x_jx_k\), impose the affine side constraint
\[
W_{u_j u_k}-W_{e_1 u_i}-1=0.
\]

We now verify correctness.

\paragraph{Soundness.}
Suppose \(\mathcal{I}\) has a satisfying assignment
\[
x_1,\dots,x_n\in\mathbb{R}_{\ge 0}.
\]
Define
\[
h_{e_a}=e_a^{(r)}
\qquad \text{for } a=1,\dots,r,
\]
and
\[
h_{u_i}=(x_i,1,0,\dots,0)
\qquad \text{for } i=1,\dots,n.
\]
These vectors lie in \(\mathbb{R}_+^r\). The anchor constraints hold by
construction. The variable-row constraints hold because
\[
\langle e_2^{(r)},(x_i,1,0,\dots,0)\rangle = 1
\]
and
\[
\langle e_t^{(r)},(x_i,1,0,\dots,0)\rangle = 0
\qquad \text{for } t=3,\dots,r.
\]
Moreover,
\[
W_{e_1 u_i}=x_i,
\]
and
\[
W_{u_j u_k}=x_jx_k+1.
\]
Hence the constant, addition, and multiplication constraints are satisfied
exactly as in the rank-$2$ construction.

\paragraph{Completeness.}
Conversely, suppose the constructed instance has a feasible realization.
By Lemma~\ref{lem:anchor-frame-rigidity}, we may normalize so that
\[
h_{e_a}=e_a^{(r)}
\qquad \text{for all } a=1,\dots,r.
\]
The variable-row constraints then force
\[
h_{u_i}=(x_i,1,0,\dots,0)
\]
for some uniquely determined \(x_i\in\mathbb{R}_{\ge 0}\), as shown above.
Consequently,
\[
W_{e_1 u_i}=x_i,
\qquad
W_{u_j u_k}=x_jx_k+1.
\]
The affine side constraints therefore enforce exactly the source arithmetic
constraints \(x_i=1\), \(x_i=x_j+x_k\), and \(x_i=x_jx_k\). Thus the extracted
tuple \((x_1,\dots,x_n)\) is a satisfying assignment of \(\mathcal{I}\).

Therefore the reduction is correct. Since \(r\) is fixed, the construction
uses \(n+r\) indices and \(O(n+m+r^2)\) specified entries and affine
constraints, where \(m\) is the number of source constraints. Hence the
construction is polynomial-time. It follows that constrained rank-$r$
nonnegative Gram feasibility is \(\exists\mathbb{R}\)-hard.
\end{proof}

\begin{theorem}
\label{thm:classification-fixed}
For every fixed integer \(r\ge 2\), constrained rank-$r$ nonnegative Gram
feasibility is \(\exists\mathbb{R}\)-complete.
\end{theorem}

\begin{proof}
Membership in \(\exists\mathbb{R}\) follows from
Proposition~\ref{prop:er-membership}. Hardness follows from
Theorem~\ref{thm:hardness-fixed-r}.
\end{proof}
\section{Open Problems}
\label{sec:open}

The preceding sections establish that constrained nonnegative Gram
feasibility is $\exists\mathbb{R}$-complete for every fixed rank
$r \ge 2$. This raises several natural questions about related
factorization problems.

\subsection{Unconstrained symmetric nonnegative factorization}

The hardness result of Section~\ref{sec:reduction} relies on affine
side constraints on selected entries of the Gram matrix. A fundamental
question is whether similar hardness persists when these constraints
are removed.

\begin{definition}
Given a symmetric rational matrix $W\in\mathbb{Q}^{N\times N}$ and an integer
$r$, the \emph{unconstrained symmetric nonnegative factorization feasibility
problem} asks whether there exists
\[
H\in\mathbb{R}_+^{N\times r}
\]
such that
\[
W = HH^\top .
\]
\end{definition}

Determining the complexity of this problem remains open.

\subsection{Approximate factorization}

In many applications, symmetric nonnegative factorization is posed as
an optimization problem rather than an exact feasibility question.
A natural decision formulation asks whether an approximate factorization
exists within a given tolerance.

\begin{definition}
Given a symmetric rational matrix $W\in\mathbb{Q}^{N\times N}$, an integer $r$,
and a rational threshold $\tau\in\mathbb{Q}_{\ge 0}$, determine whether there exists
$H \in \mathbb{R}_+^{N\times r}$ such that $\|W - HH^\top\|_F^2 \le \tau$
\end{definition}

Understanding the complexity of this approximate feasibility problem
would clarify whether the hardness phenomena identified in this paper
persist under perturbations of the exact equality constraints.

\subsection{Restricted constraint families}

The reduction of Section~\ref{sec:reduction} uses affine relations
among selected Gram entries in order to encode arithmetic operations.
This raises the question of whether the problem becomes easier when
the constraint structure is restricted.

\begin{problem}
Identify classes of constrained Gram feasibility instances that are
solvable in polynomial time. Examples include cases in which
\begin{itemize}[leftmargin=2em]
\item only diagonal entries of the Gram matrix are specified,
\item only a bounded number of Gram entries are constrained, or
\item affine constraints are replaced by inequalities.
\end{itemize}
\end{problem}

\subsection{Connections with other matrix factorization problems}

Matrix factorization problems arise in many areas of optimization and
computational geometry. Classical examples include nonnegative matrix
factorization and completely positive matrix factorization.

\begin{problem}
Clarify the precise relationships between constrained nonnegative
Gram feasibility and other factorization problems, such as
nonnegative rank and completely positive factorizations.
\end{problem}

\section{Conclusion}
\label{sec:conclusion}

We proved that constrained nonnegative Gram feasibility is
$\exists\mathbb{R}$-complete already in rank~$2$, and more generally for every
fixed rank $r\ge 2$. Thus constrained symmetric nonnegative Gram factorization
belongs to the class of semialgebraic feasibility problems whose exact solution
is complete for the existential theory of the reals
\cite{schaefer2010complexity,cardinal2013survey}.

The main technical contribution is a rank-$2$ geometric encoding of arithmetic
within the nonnegative orthant. Once suitable anchor rows are fixed, variables
can be represented on an affine slice of $\mathbb{R}_+^2$, and their algebraic
relations can be enforced through affine constraints on Gram entries. This
shows that low-rank nonnegative Gram geometry already has the expressive power
required to simulate existential-real arithmetic.

At the same time, the scope of the result is precise. The hardness proof relies
essentially on affine side constraints among selected Gram entries, and
therefore does not resolve the complexity of the unconstrained symmetric
nonnegative factorization feasibility problem
\[
W = HH^\top, \qquad H \in \mathbb{R}_+^{N\times r}.
\]
Determining whether the unconstrained problem is also
$\exists\mathbb{R}$-hard remains the most immediate open question. Other
natural directions include approximate feasibility, restricted constraint
families, and the relation between constrained nonnegative Gram feasibility and
other low-rank factorization problems.

We hope that the perspective developed here helps clarify the boundary between
tractable and existential-real matrix factorization problems.

\printbibliography

@book{berman2003cp,
  title     = {Completely Positive Matrices},
  author    = {Berman, Abraham and Shaked-Monderer, Naomi},
  year      = {2003},
  publisher = {World Scientific},
  address   = {Singapore},
  doi       = {10.1142/5270}
}

@article{bomze2018copositive,
  title     = {Copositive optimization -- recent developments and applications},
  author    = {Bomze, Immanuel M.},
  journal   = {European Journal of Operational Research},
  volume    = {216},
  number    = {3},
  pages     = {509--520},
  year      = {2012},
  publisher = {Elsevier},
  doi       = {10.1016/j.ejor.2011.04.026}
}

@article{cohen1993nonnegativerank,
  title   = {Nonnegative ranks, decompositions, and factorizations of nonnegative matrices},
  author  = {Cohen, Joel E. and Rothblum, Uriel G.},
  journal = {Linear Algebra and its Applications},
  volume  = {190},
  pages   = {149--168},
  year    = {1993},
  doi     = {10.1016/0024-3795(93)90226-J}
}

@article{vavasis2009complexity,
  title   = {On the Complexity of Nonnegative Matrix Factorization},
  author  = {Vavasis, Stephen A.},
  journal = {SIAM Journal on Optimization},
  volume  = {20},
  number  = {3},
  pages   = {1364--1377},
  year    = {2010},
  doi     = {10.1137/070709967}
}

@article{lee1999nmf,
  title   = {Learning the parts of objects by non-negative matrix factorization},
  author  = {Lee, Daniel D. and Seung, H. Sebastian},
  journal = {Nature},
  volume  = {401},
  number  = {6755},
  pages   = {788--791},
  year    = {1999},
  doi     = {10.1038/44565}
}

@article{fawzi2015psdrank,
  title   = {Positive semidefinite rank},
  author  = {Fawzi, Hamza and Gouveia, Jo\~{a}o and Parrilo, Pablo A. and Robinson, Richard and Thomas, Rekha R.},
  journal = {Mathematical Programming},
  volume  = {153},
  number  = {1},
  pages   = {133--177},
  year    = {2015},
  doi     = {10.1007/s10107-014-0818-z}
}

@inproceedings{schaefer2010complexity,
  title     = {Complexity of Some Geometric and Topological Problems},
  author    = {Schaefer, Marcus},
  booktitle = {Graph Drawing: 17th International Symposium, GD 2009, Chicago, IL, USA, September 22-25, 2009, Revised Papers},
  series    = {Lecture Notes in Computer Science},
  volume    = {5849},
  pages     = {334--344},
  year      = {2010},
  publisher = {Springer Berlin Heidelberg},
  doi       = {10.1007/978-3-642-11805-0_32}
}

@article{cardinal2013survey,
  title         = {The Existential Theory of the Reals as a Complexity Class: {A} Compendium},
  author        = {Schaefer, Marcus and Cardinal, Jean and Miltzow, Tillmann},
  journal       = {arXiv preprint arXiv:2407.18006},
  year          = {2024},
  eprint        = {2407.18006},
  archivePrefix = {arXiv},
  primaryClass  = {cs.CC},
  doi           = {10.48550/arXiv.2407.18006}
}

@incollection{mnev1988universality,
  title     = {The universality theorems on the classification problem of configuration varieties and convex polytopes varieties},
  author    = {Mn{\"e}v, Nikolai E.},
  booktitle = {Topology and Geometry: Rohlin Seminar},
  series    = {Lecture Notes in Mathematics},
  volume    = {1346},
  pages     = {527--544},
  year      = {1988},
  publisher = {Springer-Verlag},
  address   = {Berlin, Heidelberg},
  doi       = {10.1007/BFb0082792}
}

@book{richter1996oriented,
  author    = {Richter-Gebert, J{\"u}rgen},
  title     = {Realization Spaces of Polytopes},
  series    = {Lecture Notes in Mathematics},
  volume    = {1643},
  year      = {2006},
  publisher = {Springer Berlin Heidelberg},
  address   = {Berlin, Heidelberg},
  doi       = {10.1007/BFb0093761}
}

@inproceedings{arora2012nmf,
  title     = {Computing a Nonnegative Matrix Factorization -- Provably},
  author    = {Arora, Sanjeev and Ge, Rong and Kannan, Ravindran and Moitra, Ankur},
  booktitle = {Proceedings of the Forty-fourth Annual ACM Symposium on Theory of Computing (STOC 2012)},
  pages     = {145--162},
  year      = {2012},
  publisher = {Association for Computing Machinery},
  address   = {New York, NY, USA},
  doi       = {10.1145/2213977.2214002}
}

@article{laurent2015cp,
  title     = {On the completely positive and positive-semidefinite-preserving cones},
  author    = {Barker, G. P. and Hill, Richard D. and Haertel, Raymond D.},
  journal   = {Linear Algebra and its Applications},
  volume    = {56},
  pages     = {221--229},
  year      = {1984},
  publisher = {Elsevier},
  doi       = {10.1016/0024-3795(84)90124-7}
}

@article{Schaefer2017Fixed,
  author  = {Schaefer, Marcus and {\v{S}}tefankovi{\v{c}}, Daniel},
  title   = {Fixed Points, Nash Equilibria, and the Existential Theory of the Reals},
  journal = {Theory of Computing Systems},
  volume  = {60},
  number  = {2},
  pages   = {172--193},
  year    = {2017},
  doi     = {10.1007/s00224-015-9662-0}
}

@article{Schaefer2018Tensor,
  author  = {Schaefer, Marcus and {\v{S}}tefankovi{\v{c}}, Daniel},
  title   = {The Complexity of Tensor Rank},
  journal = {Theory of Computing Systems},
  volume  = {62},
  number  = {5},
  pages   = {1161--1174},
  year    = {2018},
  doi     = {10.1007/s00224-017-9800-y}
}

@article{Schaefer2024Beyond,
  author  = {Schaefer, Marcus},
  title   = {Beyond the Existential Theory of the Reals},
  journal = {Theory of Computing Systems},
  volume  = {68},
  number  = {2},
  pages   = {195--226},
  year    = {2024},
  doi     = {10.1007/s00224-023-10151-x}
}

@article{Foerster2026Thickness,
  author  = {F{\"o}rster, Helena and Miltzow, Tillmann and Schnider, Patrick},
  title   = {Geometric Thickness of Multigraphs is \(\exists \mathbb{R}\)-Complete},
  journal = {Algorithmica},
  volume  = {88},
  number  = {3},
  pages   = {1--38},
  year    = {2026},
  doi     = {10.1007/s00453-025-01351-7}
}

@InProceedings{Abrahamsen2025Embeddability,
  author    = {Abrahamsen, Mikkel and Kleist, Linda and Miltzow, Tillmann},
  title     = {Geometric Embeddability of Complexes Is $\exists\mathbb{R}$-Complete},
  booktitle = {39th International Symposium on Computational Geometry (SoCG 2023)},
  pages     = {1:1--1:19},
  series    = {Leibniz International Proceedings in Informatics (LIPIcs)},
  ISBN      = {978-3-95977-281-5},
  ISSN      = {1868-8969},
  year      = {2023},
  volume    = {258},
  editor    = {Barequet, Gill and Tóth, Csaba D.},
  publisher = {Schloss Dagstuhl -- Leibniz-Zentrum f{\"u}r Informatik},
  address   = {Dagstuhl, Germany},
  URL       = {https://drops.dagstuhl.de/entities/document/10.4230/LIPIcs.SoCG.2023.1},
  DOI       = {10.4230/LIPIcs.SoCG.2023.1}
}

@article{Erickson2025Canonical,
  title     = {Framework for $\exists\mathbb{R}$-Completeness of Two-Dimensional Packing Problems},
  author    = {Abrahamsen, Mikkel and Miltzow, Tillmann and Seiferth, Nadja},
  journal   = {TheoretiCS},
  volume    = {3},
  pages     = {12:1--12:53},
  year      = {2024},
  publisher = {TheoretiCS Foundation},
  doi       = {10.46298/theoretics.2024.11145},
  eprint    = {2004.07090},
  archivePrefix = {arXiv}
}

\end{document}